\documentclass[11pt]{amsart}
\usepackage{amsmath,amssymb,microtype}
\usepackage[letterpaper,hmargin=1.25in,vmargin=1.15in]{geometry}
\usepackage[hidelinks]{hyperref}
\hypersetup{pdftitle={Diagonal Rational Partitions},
  pdfauthor={K. Srinivasa Raghava}}
\allowdisplaybreaks[1]
\newtheorem{theorem}{Theorem}[section]
\newtheorem{proposition}[theorem]{Proposition}
\newtheorem{lemma}[theorem]{Lemma}
\newtheorem{corollary}[theorem]{Corollary}
\newtheorem{conjecture}[theorem]{Conjecture}
\theoremstyle{remark}
\newtheorem{remark}[theorem]{Remark}
\DeclareMathOperator{\Var}{Var}
\DeclareMathOperator{\lcm}{lcm}

\DeclareMathOperator{\li}{li}
\newcommand{\Pp}{\mathbb P}
\newcommand{\Ee}{\mathbb E}
\newcommand{\Z}{\mathbb Z}
\newcommand{\CN}{\mathcal C_N}

\title{Diagonal Rational Partitions}
\author{K. Srinivasa Raghava}
\address{Pie Mathematics Association}
\email{srinivasaraghavak@gmail.com}
\dedicatory{To the lasting memory of Hardy and Ramanujan, in honour of their work on integer partitions}

\begin{document}
\begin{abstract}
Let \(R_N(m)\) count unordered partitions of \(m\) into nonintegral reduced fractions whose positive numerators and denominators are at most \(N\). For \(m\asymp N\), a variational formula for \(\log R_N(m)\) expresses the correction due to integrality as a prime sum, with error \(O(N^{5/4}/\log N)\). It gives the terms of orders \(N^{3/2}/\log N\) and \(N^{3/2}/\log^2 N\) after the leading term. The first correction coefficient has second-derivative jumps on a dense set. A counting identity gives two-term asymptotics for consecutive-count ratios in two ranges of \(m/N\) and simultaneous concentration of the sums of parts with prime denominators above \(N/2\). We obtain joint limits for the denominator and size under size-biased sampling and uniform sampling among distinct parts. For fixed \(N\), \(R_N(m)\) is a quasipolynomial with an explicit leading coefficient. The normalized numerator coefficients of its generating function satisfy a central limit theorem, and we determine the first moment that differs from that of the corresponding sum of independent residues.
\end{abstract}
\maketitle

\section{Introduction}
Fix \(N\ge2\) and set
\[
 \mathcal C_N=\{a/b:1\le a\le N,\ 2\le b\le N,\ (a,b)=1\}.
\]
For a vector \(x\) of nonnegative integer multiplicities indexed by
\(\mathcal C_N\), write \(S(x)=\sum_{\alpha\in\mathcal C_N}\alpha x_\alpha\).
A rational partition of \(m\in\Z_{\ge0}\) is such a vector with \(S(x)=m\).
Write \(R_N(m)\) for their number, \(\mathbb U_{N,m}\) for their uniform
law, and \(R(n)=R_n(n)\) for the diagonal count. For example,
\(R(2),\ldots,R(8)=1,16,226,8047,223310,19427673,1208928320\).
Put \(L_k=\lcm(1,\ldots,k)\) and \(\Phi_k=\sum_{j\le k}\phi(j)\),
where \(\phi\) is Euler's totient with \(\phi(1)=1\). For each prime \(p\le N\),
let \(q_p\) be the largest power of \(p\) at most \(N\).

For fixed \(\rho>0\), the integer-total constraint contributes a negative
term of order \(N^{3/2}/\log N\) to \(\log R_N(\lfloor\rho N\rfloor)\).
For each prime \(p>\sqrt N\), the parts whose denominators are divisible
by \(p\) give a separate congruence. Their independence under a product
measure gives the upper estimate. For the lower estimate we add copies
of reserved fractions to satisfy the remaining congruences and reach
the prescribed total. The method gives the corrections to the leading
term \(\sqrt{2\rho}\,N^{3/2}\) and an estimate for consecutive ratios.
On the diagonal, the sums of parts of denominator \(p>N/2\) concentrate
at \(2\) or \(3\) on opposite sides of a threshold near \(0.7522N\).

The product measure follows the conditioning method for ordinary
partitions \cite{fristedt,erdos-lehner,granovsky-stark-erlihson,vershik,duchon-et-al}.
For large prime denominators the local count uses Hardy--Ramanujan
\cite{hardy-ramanujan}. For related asymptotic methods and integer-part
restrictions see \cite{ingham-tauberian,meinardus,roth-szekeres,vaughan}.
Hoelscher and Palsson \cite{hoelscher-palsson} restrict the numerators
and use a common denominator. Work on unit fractions includes
\cite{croot,bloom,bloom-elsholtz,conlon-et-al,liu-sawhney,bettin-et-al}.
Here reduced fractions can repeat and can have any numerator up to \(N\).

All logarithms are natural. The indices \(p,\ell\) in sums and products
range over primes unless stated otherwise, and \(\pi(x)\) counts primes.
Uniform constants may depend on the compact intervals in the statements.

\section{The Boltzmann measure and large primes}\label{sec:boltzmann}

For \(\beta>0\), under \(\mathbb P_\beta\) the multiplicities are independent with
\(\Pp_\beta(X_\alpha\ge j)=e^{-\beta\alpha j}\) for \(j\in\Z_{\ge0}\).
Put \(f_0(x)=-\log(1-e^{-x})\) and
\(F_N(\beta)=\sum_{\alpha\in\CN}f_0(\beta\alpha)\). A configuration of
total \(S\) has probability \(e^{-F_N(\beta)-\beta S}\), so
\begin{equation}\label{eq:boltzmann}
 R_N(m)=e^{F_N(\beta)+\beta m}\Pp_\beta(S=m),\qquad
 \mathbb U_{N,m}=\Pp_\beta(\,\cdot\,\mid S=m).
\end{equation}

\begin{lemma}\label{lem:bulk}
Uniformly for \(\beta/\sqrt N\) in a compact positive interval,
\begin{align*}
 F_N(\beta)&=\frac{N^2}{2\beta}+O(N\log^2 N),\\
 \Ee_\beta S&=\frac{N^2}{2\beta^2}+O(\sqrt N\log N),&
 \Var_\beta S&=\frac{N^2}{\beta^3}+O(\log N).
\end{align*}
\end{lemma}
\begin{proof}
Together with \(f_0\), use \(f_1(x)=x/(e^x-1)\) and
\(f_2(x)=(x/(2\sinh(x/2)))^2\). Their integrals on \((0,\infty)\)
are \(\zeta(2),\zeta(2),2\zeta(2)\), by termwise integration of
series in \(e^{-x}\). Each is decreasing, so
\[
 \left|\sum_{a\ge1} f_i(sa)-s^{-1}\int_0^\infty f_i(x)\,dx\right|
 \le\int_0^1 f_i(sx)\,dx .
\]
The right side is \(O(1+\max(0,\log(1/s)))\) for \(i=0\) and at most
\(1\) otherwise. Extending \(a\le N\) to \(a\ge1\) in the coprime
sums costs \(O(1)\) in total, since \(f_i(x)\ll e^{-x/2}\) for \(x\ge1\)
and each denominator contributes \(O(N e^{-\beta/2})\).
M\"obius inversion in \((a,b)=1\), with M\"obius function \(\mu\), and
\(\sum_{d\mid b}\mu(d)/d=\phi(b)/b\) give the main term
\((\Phi_N-1)\beta^{-1}\int f_i\) for \(\sum_{a/b\in\CN}f_i(\beta a/b)\).
The errors are \(O(N\log^2 N)\) for \(i=0\) and \(O(N\log N)\) otherwise,
since \(\sum_{b\le N}\sum_{d\mid b}1=O(N\log N)\).
The three sums, multiplied by \(1,\beta^{-1},\beta^{-2}\), give
\(F_N,\Ee_\beta S,\Var_\beta S\). Use
\(\Phi_N=3N^2/\pi^2+O(N\log N)\) \cite{hardy-wright}.
\end{proof}

For a prime \(\sqrt N<p\le N\), put \(k=\lfloor N/p\rfloor\) and let
\(B_p\) consist of fractions with denominator \(jp\), \(1\le j\le k\).
Write \(V_p=\sum_{\alpha\in B_p}\alpha X_\alpha\) and
\(E_p=\{V_p\in L_k^{-1}\Z\}\). The denominator of \(V_p\) divides \(pL_k\), with \(p\nmid L_k\);
that of \(S-V_p\) is prime to \(p\). Thus \(E_p\) holds exactly when
the reduced denominator of \(S\) is prime to \(p\).
The sets \(B_p\) are disjoint. For sufficiently large \(N\), define
\begin{equation}\label{eq:cutoff}
 K=\lfloor\log N\rfloor,\qquad
 \mathcal S_N=\{p\text{ prime}:N/(K+1)<p\le N\},\qquad
 \mathcal E_N=\bigcap_{p\in\mathcal S_N}E_p .
\end{equation}
For large \(N\), \(p\in\mathcal S_N\) implies \(p>\sqrt N\) and \(k\le K\).

Put \(\theta_k^2=\zeta(2)\Phi_k\) and \(I_k=(1/(k+1),1/k]\).
For \(t\in I_k\) and \(u>0\) define
\begin{equation}\label{eq:G}
 G(u,t)=\min_{w\in L_k^{-1}\Z_{\ge0}}
   \left(\sqrt{uw}-\theta_k\sqrt{t/u}\right)^2
   =uD_k(\theta_k\sqrt t/u)^2 ,
\end{equation}
where \(D_k(y)\) is the distance from \(y\) to
\(\{\sqrt{r/L_k}:r\in\Z_{\ge0}\}\).
The set contains \(0\), and the gap between its members bracketing
\(y>0\) is at most \((L_k y)^{-1}\). Hence
\begin{equation}\label{eq:grid}
 D_k(y)\le\min\{y,(2L_ky)^{-1}\}.
\end{equation}
On \(I_k\), \(\theta_k^2t\asymp k\) by the totient estimate. Thus for
each compact \(U\subset(0,\infty)\),
\begin{equation}\label{eq:G-bounds}
 G(u,t)\ll_U(kL_k^2)^{-1}\quad(u\in U),\qquad
 0\le G(u,t)\le \pi/\sqrt{12}\quad(u>0,\ 0<t\le1).
\end{equation}
For the global bound take the weighted geometric mean with weights
\(3/4,1/4\) of \(\theta_k^2t/u\) and \(u^3/(4L_k^2\theta_k^2t)\),
the two bounds from \eqref{eq:grid}.
It equals \(\theta_k\sqrt{t/(2L_k)}\le\pi/\sqrt{12}\) by
\(\Phi_k\le k^2\), \(t\le1/k\) and \(L_k\ge k\).

\section{A uniform estimate for one prime}\label{sec:block-estimate}

For \(p\in\mathcal S_N\) let \(F_{N,p}(\gamma)=
\sum_{\alpha\in B_p} f_0(\gamma\alpha)\) and \(\lambda_p=p\theta_k^2\),
where \(k=\lfloor N/p\rfloor\).
\begin{lemma}\label{lem:block-energy}
Uniformly for \(\log^2 N\le\gamma\le p/k\),
\begin{equation}\label{eq:block-energy}
 F_{N,p}(\gamma)=\lambda_p/\gamma+O(\log N).
\end{equation}
For \(\gamma\ge\log^2 N\), uniformly in \(p\in\mathcal S_N\),
\begin{equation}\label{eq:block-moments}
 \Ee_\gamma V_p=\lambda_p/\gamma^2+O(k^2/\gamma),\qquad
 \Var_\gamma V_p=2\lambda_p/\gamma^3+O(k^2/\gamma^2).
\end{equation}
\end{lemma}
\begin{proof}
The numerator tails are \(O(1)\) as in Lemma~\ref{lem:bulk}.
The integral comparisons for \(f_1,f_2\) have total error
\(O(\sum_{j\le k}\sum_{d\mid jp}1)=O(k^2)\), and
\(\sum_{j\le k}\phi(jp)=(p-1)\Phi_k\). This gives \eqref{eq:block-moments}.
For \(g(s)=\sum_{a\ge1}f_0(as)\), the eta transformation at
\(\tau=is/(2\pi)\) gives \cite[Ch.~3]{apostol-modular}
\begin{equation}\label{eq:eta}
 g(s)=\frac{\pi^2}{6s}+\frac12\log\frac{s}{2\pi}-\frac{s}{24}+g(4\pi^2/s).
\end{equation}
In particular \(g(s)=\zeta(2)/s+\tfrac12\log(s/(2\pi))+O(s)\)
for \(s\le1\) and \(g(s)\ll e^{-s}\) for \(s\ge1\).
M\"obius inversion gives
\[
 F_{N,p}(\gamma)=
 \sum_{j\le k}\sum_{d\mid j}\mu(d)
  \bigl(g(\gamma d/(jp))-g(\gamma d/j)\bigr)+O(1).
\]
The second arguments are at least \(\log N\), and their terms total
\(O(k^2/N)\). The first arguments are at most \(1\), with sum
\(O(\gamma k/p)=O(1)\) since
\(\sum_{j\le k}\sum_{d\mid j} d/j\le k\zeta(2)\).
The reciprocal terms give \(\lambda_p/\gamma\) and the logarithms
give \(\tfrac12\log(\gamma/(2\pi p))-\tfrac12\log L_k=O(\log N)\).
With \(\Lambda\) the von Mangoldt function, \(\sum_{d\mid j}\mu(d)\log d=-\Lambda(j)\) and
\(\log L_k=\sum_{j\le k}\Lambda(j)\sim k\)
\cite[(6.12)]{montgomery-vaughan}.
\end{proof}

\begin{lemma}\label{lem:completion}
For \(k\ge1\), each integer \(d\ge\pi(k)L_k\) is a nonnegative integer
combination of \(L_k/j\), \(1\le j\le k\).
\end{lemma}
\begin{proof}
For each prime dividing \(L_k\), let \(q\) be its highest power in \(L_k\).
For each such \(q\),
choose \(0\le n_q<q\) so that \(n_q L_k/q\equiv d\pmod q\).
Then \(d-\sum_q n_qL_k/q\) is a nonnegative multiple of \(L_k\),
which is the generator for \(j=1\).
\end{proof}

\begin{proposition}\label{prop:block}
Let \(u\) range over a compact positive interval and put \(\beta=u\sqrt N\).
Uniformly for \(p\in\mathcal S_N\),
\begin{align}
 |\log \Pp_\beta(E_p)+\sqrt N\,G(u,p/N)|
   &\ll \sqrt{k}\,N^{1/4}+k\log N+\log^3 N,\label{eq:local}\\
 \Ee_\beta(V_p^2\mid E_p)&\ll k^2.\label{eq:block-second}
\end{align}
\end{proposition}

\begin{proof}
Write \(F=F_{N,p}\), \(\lambda=\lambda_p\), \(w_0=\lambda/\beta^2\asymp k\)
and \(M=\sqrt N\,G(u,p/N)\). Put \(H(w)=\beta(\sqrt w-\sqrt{w_0})^2\),
whose minimum on \(L_k^{-1}\Z_{\ge0}\) is \(M\).
Changing the parameter on the coordinates in \(B_p\) gives
\begin{equation}\label{eq:block-tilt}
 \Pp_\beta(V_p=w)=
 e^{F(\gamma)-F(\beta)+(\gamma-\beta)w}\Pp_\gamma(V_p=w).
\end{equation}
For \(0<w\le16w_0\) choose \(\gamma=\min(\sqrt{\lambda/w},p/k)\).
The unconstrained choice minimizes \(\lambda/\gamma+\gamma w\)
with value \(2\sqrt{\lambda w}\). If the choice is capped, the increase
is \((\sqrt{\lambda/\gamma}-\sqrt{\gamma w})^2\le\lambda/\gamma=O(k^3)\).
In either case \(\gamma\ge\beta/4\) for large \(N\), so
\eqref{eq:block-energy} applies. At \(w=0\) use the empty configuration.
Thus for \(w\in L_k^{-1}\Z_{\ge0}\) with \(w\le16w_0\),
\begin{equation}\label{eq:bounded-block-upper}
 \Pp_\beta(V_p=w)\le e^{-H(w)+O(\log^3 N)}.
\end{equation}
For \(w>16w_0\), taking \(\gamma=\beta/2\) in \eqref{eq:block-tilt}
gives \(\Pp_\beta(V_p=w)\le e^{-\beta w/4+O(\log N)}\).
Since \(\log L_k=O(k)\), the number of allowed values up to \(16w_0\)
and \((1-e^{-\beta/(4L_k)})^{-1}\) are polynomial in \(N\).
Summing the bounds and using \(M\le\beta w_0\) gives
\(\Pp_\beta(E_p)\le e^{-M+O(\log^3 N)}\).

Let \(w_*\) minimize \(H\) on \(L_k^{-1}\Z_{\ge0}\). When \(w_*=0\),
\(e^{-F(\beta)}=e^{-M+O(\log N)}\) suffices.
Otherwise \(w_0/4\le w_*\le4w_0\). The lower bound follows by comparing
with \(4w_*\) if \(w_*<w_0/4\), and the upper bound by comparing with \(0\).
Remove \(1/(jp)\), \(1\le j\le k\), and denote the remaining sum by \(V'\).
For \(\gamma\) between \(\gamma_0/2\) and \(2\gamma_0\), where
\(\gamma_0=\sqrt{\lambda/w_*}\asymp\sqrt N\), \eqref{eq:block-moments}
gives mean \(\lambda/\gamma^2+O(k^2/\sqrt N)\) and variance at most
\(v=Ck/\sqrt N\) for a fixed \(C\). Each removed mean is at most \(1/\gamma\).
Put \(\delta=4\sqrt v+\pi(k)/p\). The endpoint means straddle the
midpoint of \([w_*-\delta,w_*-\pi(k)/p]\) since \(w_*\asymp k\);
choose \(\gamma\) with that mean. Chebyshev's inequality puts \(V'\)
in the interval with probability at least \(3/4\).
There \(pL_k(w_*-V')\) is an integer at least \(\pi(k)L_k\).
By Lemma~\ref{lem:completion} add copies of the removed fractions
using the lexicographically least multiplicity vector. This map
into totals \(w_*\) is injective by deletion.

The log normalizing constant of the remaining coordinates is
\(\lambda/\gamma+O(k\log N)\), since each removed term is \(O(\log N)\)
by \(jp\le N\). Their configurations in the interval therefore number
at least \(\tfrac34\exp(\lambda/\gamma+\gamma(w_*-\delta)-O(k\log N))\).
Use \(\gamma\delta=O(\sqrt k\,N^{1/4})\) and
\(\lambda/\gamma+\gamma w_*\ge2\sqrt{\lambda w_*}\) to obtain, at \(\beta\),
\begin{equation}\label{eq:winning-block-lower}
 \Pp_\beta(V_p=w_*)\ge e^{-M-O(\sqrt{k}N^{1/4}+k\log N)} .
\end{equation}
This also holds at zero and proves \eqref{eq:local}.
Finally \(E_p\) contains the empty configuration, so exponential
Markov gives \(\Pp_\beta(V_p>x\mid E_p)\le
e^{F(\beta/2)-\beta x/2}\). Since \(F(\beta/2)=O(k\sqrt N)\),
this is at most \(e^{-\beta x/4}\) for \(x\ge C'k\).
Integrate \(2x\,\Pp_\beta(V_p>x\mid E_p)\) to prove \eqref{eq:block-second}.
\end{proof}

\section{Asymptotic enumeration}\label{sec:enumeration}

Write \(u_\rho=(2\rho)^{-1/2}\), \(h_\rho(u)=1/(2u)+\rho u\) and
\(P_N(u)=\sum_{p\le N}G(u,p/N)\). The coefficients below are defined by
\begin{equation}\label{eq:kappas}
\begin{aligned}
 J(u)&=\int_0^1G(u,t)\,dt,& J_1(u)&=\int_0^1G(u,t)\log(1/t)\,dt,\\
 \kappa(\rho)&=J(u_\rho),&
 \kappa_2(\rho)&=J_1(u_\rho)+\tfrac12 u_\rho^3 J'(u_\rho)^2 .
\end{aligned}
\end{equation}

\begin{theorem}\label{thm:enumeration}
The derivative \(J'\) is locally Lipschitz. For \(\rho\) in compact positive
intervals and \(m=\lfloor\rho N\rfloor\), uniformly
\begin{align}
 \log R_N(m)&=\min_{u>0}\{N^{3/2}h_\rho(u)-\sqrt N\, P_N(u)\}
                  +O(N^{5/4}/\log N),\label{eq:prime-variational}\\
 \frac{\log R_N(m)}{N^{3/2}}&=\sqrt{2\rho}
        -\frac{\kappa(\rho)}{\log N}-\frac{\kappa_2(\rho)}{\log^2 N}
        +O((\log N)^{-3}).\label{eq:three-terms}
\end{align}
\end{theorem}

\subsection{The saddle point}

By independence and the product law,
\begin{equation}\label{eq:psi-saddle}
 \Psi_{N,m}(\beta)=\beta m+F_N(\beta)+\sum_{p\in\mathcal S_N}
   \log\Pp_\beta(E_p)
 =\log\sum_{x\in\mathcal E_N}e^{\beta(m-S(x))}.
\end{equation}
Exponential decay justifies differentiation on compact positive
\(\beta\)-intervals. Its first two derivatives are
\(m-\Ee_\beta(S\mid\mathcal E_N)\) and \(\Var_\beta(S\mid\mathcal E_N)\).

\begin{proposition}\label{prop:saddle}
Uniformly for integers \(m\) with \(m/N\) in a compact positive interval,
the unique minimizer \(\widehat\beta_{N,m}\) of \(\Psi_{N,m}\) satisfies
\(\widehat\beta_{N,m}\asymp\sqrt N\) and
\begin{equation}\label{eq:sharp-saddle}
 \log R_N(m)=\Psi_{N,m}(\widehat\beta_{N,m})+O(N).
\end{equation}
These minimizers decrease strictly in \(m\), with
\begin{equation}\label{eq:saddle-difference}
 0<\widehat\beta_{N,m}-\widehat\beta_{N,m+1}\ll N^{-1/2}
\end{equation}
when both densities are in the interval. At \(\beta=\widehat\beta_{N,m}\),
for every event \(A\),
\begin{equation}\label{eq:transfer}
 \mathbb U_{N,m}(A)\le e^{CN}\Pp_\beta(A\mid\mathcal E_N).
\end{equation}
\end{proposition}

\begin{proof}
Brun--Titchmarsh \cite[(1.12)]{montgomery-vaughan-sieve} gives
\begin{equation}\label{eq:prime-shell-count}
 \#\{p:N/(k+1)<p\le N/k\}\ll N/(k^2\log N)\qquad(1\le k\le K),
\end{equation}
since the interval has length \(N/(k(k+1))\). For \(\beta\asymp\sqrt N\)
the unconditional and conditional means of \(V_p\) are \(O(k)\);
their variances are \(O(k/\sqrt N)\) and \(O(k^2)\), respectively,
by \eqref{eq:block-moments} and \eqref{eq:block-second}.
Conditioning on \(\mathcal E_N\) preserves independence across \(B_p\)
and leaves the other coordinates unchanged. The total change in
the mean is \(O(N\log(K+1)/\log N)\), and the sum of conditional
variances over \(p\in\mathcal S_N\) is \(O(NK/\log N)\).
Hence Lemma~\ref{lem:bulk} gives
\begin{equation}\label{eq:conditioned-moments}
 \Ee_\beta(S\mid\mathcal E_N)=\frac{N^2}{2\beta^2}
      +O(N\log\log N/\log N),\qquad
 \Var_\beta(S\mid\mathcal E_N)\ll N.
\end{equation}
By \eqref{eq:prime-shell-count} the sum of unconditional variances
over these primes is \(O(\sqrt N\log(K+1)/\log N)=o(\sqrt N)\).
Thus the remaining coordinates already give
\(\Psi''_{N,m}(\beta)\gg\sqrt N\) on this scale.
For \(m/N\in I\), choose \(0<b_1<b_2\) with
\(1/(2b_1^2)>\max I\) and \(1/(2b_2^2)<\min I\).
The mean estimate puts the root of \(\Psi_{N,m}'\) between
\(b_1\sqrt N\) and \(b_2\sqrt N\). The variance is positive for every
\(\beta>0\), as \(\mathcal E_N\) includes the empty configuration and
two copies of \(1/2\). Strict convexity gives uniqueness, and
\(\Psi_{N,m+1}'=\Psi_{N,m}'+1\) with the lower variance bound gives
\eqref{eq:saddle-difference}.

Every configuration of total \(m\) belongs to \(\mathcal E_N\), so
\(R_N(m)\le e^{\Psi_{N,m}(\beta)}\).
For an injection giving the reverse bound, set aside \(1/2\),
\(1/q_\ell\) for \(\ell\le\sqrt N\), and \(1/(jp)\) for
\(\sqrt N<p\le N/(K+1)\), \(1\le j\le K\).
For large \(N\) they are distinct and outside the sets \(B_p\) for
\(p\in\mathcal S_N\). Their number is
\(O(1+\pi(\sqrt N)+K\pi(N/(K+1)))=O(N/\log N)\) by the prime number theorem.

Starting from a configuration in \(\mathcal E_N\) using none of these parts,
for each \(\sqrt N<p\le N/(K+1)\) write \(V_p=A/(pL_k)\).
To remove \(p\) from the reduced denominator, add \(d/(pL_K)\), where
\(d\) solves \(A+dL_k/L_K\equiv0\pmod p\) and lies in
\([\pi(K)L_K,\pi(K)L_K+p)\). The solution is unique because \(p>k\).
Lemma~\ref{lem:completion} supplies copies of \(1/(jp)\) of that total.
Use the lexicographically least vector of multiplicities.
The addition is less than \(\pi(K)/p+1/L_K\) and affects no other
prime above \(\sqrt N\). The resulting total has denominator dividing
\(Q_N=\prod_{\ell\le\sqrt N}q_\ell\). The Chinese remainder theorem
gives \(0\le n_\ell<q_\ell\) such that adding \(\sum_\ell n_\ell/q_\ell\)
makes the total an integer. The total weight added so far is at most
\[
 \Delta_N=\pi(\sqrt N)+\sum_{\sqrt N<p\le N/(K+1)}
         \left(\pi(K)/p+1/L_K\right)=O(\sqrt N).
\]
Here \(\log L_K\sim K\) implies \(L_K\ge e^{K/2}\), and the harmonic
sum bounds the term with \(\pi(K)/p\) by \(O(\log^2 N)\).
Starting below \(m-\Delta_N\) thus allows the total to be brought
to \(m\) by adding copies of \(1/2\). Deletion of all parts set aside
recovers the original configuration.

The total from the remaining fractions, conditioned on \(\mathcal E_N\),
has variance at most \(C_1N\) and mean differing from
\(\Ee_\beta(S\mid\mathcal E_N)\) by \(O(N/(\beta\log N))\), since
the omitted mean per coordinate is at most \(1/\beta\).
For \(W=4\sqrt{C_1N}\), the mean crosses \(m-\Delta_N-W/2\) as
\(\beta\) runs between \(b_1\sqrt N\) and \(b_2\sqrt N\).
At a parameter \(\beta^\dagger\) with that mean, Chebyshev's inequality
gives conditional mass at least \(3/4\) in
\([m-\Delta_N-W,m-\Delta_N]\).
The log normalizing constant on the remaining coordinates is
\(F_N-D\), with \(D=\sum_{\alpha\ \mathrm{set\ aside}}f_0(\beta^\dagger\alpha)
=O(N)\) since \(\alpha\ge1/N\).
On this interval, each configuration has conditional probability
at most \(e^{D-\Psi_{N,m}(\beta^\dagger)+\beta^\dagger(\Delta_N+W)}\).
The injection therefore gives a lower bound of three quarters of the
reciprocal of this probability bound. Since \(\Delta_N+W=O(\sqrt N)\)
and \(\Psi_{N,m}(\beta^\dagger)\ge\Psi_{N,m}(\widehat\beta_{N,m})\),
this proves \eqref{eq:sharp-saddle}.
At the minimizer, \eqref{eq:boltzmann} gives
\(\Pp_\beta(S=m\mid\mathcal E_N)=R_N(m)e^{-\Psi_{N,m}(\beta)}
\ge e^{-CN}\). Divide \(\Pp_\beta(A\cap\{S=m\}\mid\mathcal E_N)\)
by this probability to obtain \eqref{eq:transfer}.
\end{proof}

\subsection{Prime sums and expansion of the minimum}

Take a common compact interval \(U\) containing the scaled minimizers from
Proposition~\ref{prop:saddle} and all \(u_\rho\) in its interior.
By \eqref{eq:local} and \eqref{eq:prime-shell-count}, the errors
in summing \(\log\Pp_{u\sqrt N}(E_p)\) are
\(O(N^{5/4}/\log N+N\log(K+1)+N\log^2 N)\), using
\(\sum k^{-3/2}<\infty\). The terms in \(P_N(u)\) for \(p\le N/(K+1)\)
total \(O(1)\) by \eqref{eq:G-bounds} and \(L_{K+1}^2\ge N\).
Lemma~\ref{lem:bulk}, with the floor error \(O(\sqrt N)\), then gives
\begin{equation}\label{eq:saddles-close}
 \Psi_{N,m}(u\sqrt N)=N^{3/2}h_\rho(u)-\sqrt N\, P_N(u)
                              +O(N^{5/4}/\log N)\qquad(u\in U).
\end{equation}
The same holds with \(\rho=m/N\).
The minimum on the right of \eqref{eq:prime-variational} is in
the interior of \(U\) for large \(N\). Outside, \(h_\rho-h_\rho(u_\rho)\)
is uniformly bounded below by a positive constant, whereas
\(P_N(u)=O(N/\log N)\) globally. Continuity and comparison of the minima
in \eqref{eq:saddles-close} now give \eqref{eq:prime-variational}.

\begin{lemma}\label{lem:prime-summation}
With \(\li\) the logarithmic integral, let
\(\mathcal D_N=1+\sup_{2\le x\le N}|\pi(x)-\li(x)|\).
For \(2\le a<b\le N\) and real \(f\) of bounded variation on \((a,b]\),
\[
 \left|\sum_{a<p\le b}f(p)-\int_a^b f(x)\frac{dx}{\log x}\right|
 \ll\mathcal D_N(\|f\|_\infty+\operatorname{Var}_{(a,b]} f),
\]
where \(\operatorname{Var}_{(a,b]}\) is total variation.
If \(\chi:(0,1]\to\mathbb R\) satisfies
\(\sup_{I_k}|\chi|\le A/L_k\) and \(\operatorname{Var}_{I_k}\chi\le A\),
then
\begin{equation}\label{eq:general-prime-expansion}
 \frac1N\sum_{p\le N}\chi(p/N)
  =\frac{\int_0^1\chi(t)\,dt}{\log N}
    +\frac{\int_0^1\chi(t)\log(1/t)\,dt}{\log^2 N}
    +O(A/\log^3 N).
\end{equation}
\end{lemma}
\begin{proof}
On each interval in \((a,b]\) the discrepancy is at most
\(2\mathcal D_N\), irrespective of endpoint inclusion.
Integrating over level sets proves the bound for a bounded
nonnegative monotone function; write general \(f\) as its endpoint
value plus the difference of negative and positive variations on
\([x,b]\). Apply the result on \((N/(k+1),N/k]\) for \(k\le K\).
The summed errors are \(O(AK\mathcal D_N)=O(AN/\log^3 N)\) by the
prime number theorem \cite[(6.14)]{montgomery-vaughan}. The other
primes contribute \(O(A\sqrt N)\) since \(L_{K+1}\ge\sqrt N\).
For \(1/(K+1)<t\le1\), expand \(1/\log(Nt)\) to two terms in
\(1/\log N\), with remainder \(O(\log^2(1/t)/\log^3 N)\).
It is integrable against \(\chi\). Extend the two main integrals to \(0\)
using \(|\chi(t)|\le A/L_{K+1}\) on the omitted range.
\end{proof}

\begin{lemma}\label{lem:prime-expansion}
The function \(J\) is continuously differentiable with locally
Lipschitz derivative, and \(J_1\) is locally Lipschitz. For
\(\epsilon=1/\log N\), uniformly on compact positive \(u\)-intervals,
\begin{equation}\label{eq:prime-expansion}
 P_N(u)/N=\epsilon J(u)+\epsilon^2J_1(u)+O(\epsilon^3),\qquad
 \partial_u^\pm P_N(u)/N=\epsilon J'(u)+O(\epsilon^2).
\end{equation}
\end{lemma}
\begin{proof}
On \(I_k\), write \(y=\theta_k\sqrt t/u\). For a minimizing \(w\), put
\(g_w=uw-2\theta_k\sqrt{wt}+\theta_k^2t/u\).
Then \(\partial_u g_w=w-y^2\) has absolute value at most \(1/L_k\),
since \(w\) is one of the allowed values bracketing \(y^2\).
For fixed \(t\), only finitely many \(w\) can minimize on a compact
\(u\)-interval, since \(y\) is bounded there. For fixed \(u\), ties occur
at countably many \(t\). Dominated convergence yields \(J'=\int_0^1\partial_u G\,dt\) and the
derivative bound also gives local Lipschitz continuity of \(J,J_1\).
On each \(I_k\), integrate \(\partial_u G=(G-2t\partial_t G)/u\) by
parts using the endpoint values from within \(I_k\). This gives
\begin{equation}\label{eq:j-endpoints}
 J'(u)=\frac{3J(u)}{u}
 -2\sum_{k\ge1}\!\left[
 \frac{D_k(\theta_k/(u\sqrt k))^2}{k}
 -\frac{D_k(\theta_k/(u\sqrt{k+1}))^2}{k+1}\right].
\end{equation}
The summands have Lipschitz constants \(O(1/(kL_k))\) on compact positive
\(u\)-intervals. Indeed \(D_k\) is \(1\)-Lipschitz and \eqref{eq:grid}
bounds it by \(O(1/(L_k\sqrt k))\) on \(y\asymp\sqrt k\).
Summation proves the regularity of \(J'\).

Across \(I_k\), \(y^2\) increases by \(O(1)\), so
\(\partial_u^\pm G=w-y^2\) has variation \(O(1)\), including
\(O(L_k)\) jumps of size \(1/L_k\).
Also \(|\partial_t g_w|=(\theta_k^2/u)|1-\sqrt w/y|\ll k/L_k\)
for each minimizing function, giving variation \(O(1/(kL_k))\) for \(G\).
Apply Lemma~\ref{lem:prime-summation} to \(G\) and both one-sided
derivatives. For the derivatives the first integral is \(J'\)
and the second is bounded, which proves \eqref{eq:prime-expansion}.
\end{proof}

At a minimizer \(\widetilde u\) in \eqref{eq:prime-variational}, the
one-sided first-order conditions and \eqref{eq:prime-expansion} give
\(h_\rho'(\widetilde u)-\epsilon J'(\widetilde u)=O(\epsilon^2)\).
On \(U\), this function of \(\widetilde u\) has positive difference
quotients bounded below, since \(h_\rho''=u^{-3}\) and \(J'\)
is Lipschitz. Its value at \(u_\rho+\epsilon u_\rho^3J'(u_\rho)\)
is \(O(\epsilon^2)\). Thus
\begin{equation}\label{eq:discrete-optimizer}
 \widetilde u=u_\rho+\epsilon u_\rho^3J'(u_\rho)+O(\epsilon^2).
\end{equation}
Taylor expansion of \(h_\rho-\epsilon J-\epsilon^2 J_1\) at \(u_\rho\),
with displacement \eqref{eq:discrete-optimizer}, changes the value by
\(-\epsilon^2u_\rho^3J'(u_\rho)^2/2+O(\epsilon^3)\).
Together with \eqref{eq:prime-variational} and
\eqref{eq:prime-expansion} this proves \eqref{eq:three-terms}.
Both coefficients are positive since the zeros of \(G(u,\cdot)\)
form a countable set for every \(u>0\). Numerically
\(\kappa(1)\approx0.00264713\) and \(\kappa_2(1)\approx0.00107712\);
see the truncation bounds below.

For \(\rho=m/N\), \eqref{eq:saddles-close} bounds the difference between
\(\Psi_{N,m}(\sqrt N\,\widetilde u)\) and its minimum by
\(O(N^{5/4}/\log N)\). The lower bound \(\Psi''_{N,m}\gg\sqrt N\) between
the two arguments gives
\[
 |\sqrt N\,\widetilde u-\widehat\beta_{N,m}|
     =O(N^{3/8}/\sqrt{\log N}).
\]
Use \eqref{eq:discrete-optimizer} and
\(\kappa'(\rho)=-u_\rho^3J'(u_\rho)\) to obtain
\begin{equation}\label{eq:saddle-expansion}
 \widehat\beta_{N,m}/\sqrt N=u_\rho-\kappa'(\rho)/\log N
                               +O((\log N)^{-2}).
\end{equation}

\begin{corollary}\label{cor:integer-total}
Uniformly for \(u\) in compact positive intervals, at \(\beta=u\sqrt N\),
\begin{align*}
 \log\Pp_\beta(S\in\Z)&=-\sqrt N\,P_N(u)+O(N^{5/4}/\log N)\\
 &=-N^{3/2}\left(\frac{J(u)}{\log N}+\frac{J_1(u)}{\log^2 N}
                     +O((\log N)^{-3})\right).
\end{align*}
\end{corollary}
\begin{proof}
For the lower bound take \(m\) nearest the mean of \(S\) given
\(\mathcal E_N\). Then \(m/N\) is in a compact positive interval,
\(|\Psi'_{N,m}(\beta)|\le1/2\) and both \(\beta\) and
\(\widehat\beta_{N,m}\) are \(O(\sqrt N)\).
Convexity gives \(\Psi_{N,m}(\widehat\beta_{N,m})
\ge\Psi_{N,m}(\beta)-O(\sqrt N)\), so \eqref{eq:sharp-saddle} and
\eqref{eq:boltzmann} yield
\(\log\Pp_\beta(S=m)\ge\log\Pp_\beta(\mathcal E_N)-O(N)\).
For the upper bound use \(\{S\in\Z\}\subset\mathcal E_N\).
The estimates leading to \eqref{eq:saddles-close}, and then
\eqref{eq:prime-expansion}, give the two formulas.
\end{proof}
In particular \(\mathbb P_{u\sqrt N}(S\in\mathbb Z)=o(L_N^{-1})\)
uniformly on compact positive \(u\)-intervals by positivity of \(J\)
and \(\log L_N\sim N\).

\begin{corollary}\label{cor:prime-integral}
Uniformly for \(\rho\) in compact positive intervals,
\[
 \log R_N(\lfloor\rho N\rfloor)
 =\min_{u>0}\left\{N^{3/2}h_\rho(u)
    -\sqrt N\int_2^N G(u,x/N)\frac{dx}{\log x}\right\}
       +O(N^{3/2}e^{-c\sqrt{\log N}}).
\]
Under the Riemann hypothesis the error is \(O(N^{5/4}/\log N)\).
\end{corollary}
\begin{proof}
On compact positive \(u\)-intervals, \(G(u,\cdot)\) has bounded
variation by the summable bounds \(O(1/(kL_k))\) on \(I_k\) and
\(O(1/(kL_k^2))\) at the endpoints.
Lemma~\ref{lem:prime-summation}, including the \(O(1)\) term at \(p=2\),
permits replacing \(P_N(u)\) by the integral with error
\(O(\mathcal D_N)\). Unconditionally \(\mathcal D_N=O(Ne^{-c\sqrt{\log N}})\);
under RH it is \(O(\sqrt N\log N)\) \cite[(6.14) and (13.4)]{montgomery-vaughan}.
Both prime corrections are globally \(O(N/\log N)\) by
\eqref{eq:G-bounds}, so the same compact-minimum argument as in
\eqref{eq:prime-variational} applies. Substitution proves the claims.
\end{proof}

\subsection{Regularity and endpoint behavior}

\begin{proposition}\label{prop:kappa-corners}
The one-sided derivatives of \(J'\) exist at every \(u>0\).
For \(r\in\Z_{\ge0}\) put
\(\rho_r=3(\sqrt r+\sqrt{r+1})^2/(4\pi^2)\).
At \(u=u_{\rho_r}\), the right derivative of \(J'\) minus its left
derivative is \(2/u\). The function \(\kappa\) is continuously
differentiable with locally Lipschitz derivative, and its right
second derivative minus its left second derivative at \(\rho_r\)
is \(-2u_{\rho_r}^{5}\). Both sets of jumps are dense in \((0,\infty)\).
\end{proposition}
\begin{proof}
At the midpoint \(y=(s+s')/2\) of adjacent allowed values
\(s=\sqrt{r/L_k}\), \(s'=\sqrt{(r+1)/L_k}\), the derivative of \(D_k^2\)
has a jump of \(-2(s'-s)\). Since \(2y(s'-s)=1/L_k\),
the derivative of \(D_k(a/u)^2\) has jump \(-1/(uL_k)\) as \(u\)
increases through \(a/u=y\).
In \eqref{eq:j-endpoints}, \(3J(u)/u\) is continuously differentiable
and the summands have summable local Lipschitz bounds, permitting
one-sided differentiation term by term. Define
\(a_k=\Phi_kL_k/k\), \(b_k=\Phi_kL_k/(k+1)\) and
\[
 \Delta(A)=\sum_{a_k=A}\frac{1}{kL_k}
      -\sum_{b_k=A}\frac{1}{(k+1)L_k}\qquad(A\in\mathbb Q_{>0}).
\]
The midpoint contributions occur at \(u^2=4\zeta(2)A/c_j\) with
\(A=a_k\) or \(b_k\), and \(c_j=(\sqrt j+\sqrt{j+1})^2\).
For \(j\ge1\), \(c_j\) is a quadratic irrational with conjugate \(1/c_j\).
If \(c_i/c_j\) is rational for positive \(i,j\), their quadratic fields
coincide and the norm of that ratio is one. The ratio is positive,
hence one, so \(i=j\). Since \(c_0=1\), the same holds if one index is zero.
Therefore the jump at each such \(u\) is \(2\Delta(A)/u\).
At \(u=u_{\rho_r}\), \(A=1\), and \(\Delta(1)=1\) because
\(a_1=1\), \(b_1=1/2\) and \(a_k>b_k\ge k^2/(k+1)>1\) for \(k\ge2\).

The sums are absolutely convergent. For any \(T>0\),
\(\sum_{A>T}\Delta(A)>0\), since \(a_k>b_k\), the positive weight
exceeds the negative weight and \(b_k\to\infty\).
Thus arbitrarily large \(A\) have \(\Delta(A)\ne0\).
For any \(v>0\) and such \(A\to\infty\), choose
\(j=\lfloor\zeta(2)A/v\rfloor\). Since \(c_j=4j+O(1)\), the resulting
\(u^2\) tend to \(v\). Finally \(\kappa'=-u_\rho^3J'(u_\rho)\) and
\(u_\rho'=-u_\rho^3\). The decreasing map \(\rho\mapsto u_\rho\)
interchanges right and left, so the jump for \(\kappa\) is
\(-u_\rho^6\) times the jump for \(J'\).
\end{proof}

\begin{proposition}\label{prop:kappa-endpoints}
As \(\rho\downarrow0\),
\(\kappa(\rho)=\sqrt{\rho/2}\,\log\log(1/\rho)+O(\sqrt\rho)\).
As \(\rho\to\infty\),
\[
 \kappa(\rho)=c_\infty\rho^{-3/2}\bigl(1+O(\rho^{-1})\bigr),\qquad
 c_\infty=\frac{1}{16\sqrt2\,\pi^2}
      \sum_{k\ge1}\frac{\log(1+1/k)}{\Phi_k L_k^2}
      =0.0033408689\ldots .
\]
\end{proposition}
\begin{proof}
Put \(u=u_\rho\) and substitute \(y=\theta_k\sqrt t/u\):
\begin{equation}\label{eq:kappa-shell-integral}
 \int_{I_k}G(u,t)\,dt
 =\frac{2u^3}{\theta_k^2}\int_{\theta_k/(u\sqrt{k+1})}^{\theta_k/(u\sqrt{k})}
      D_k(y)^2y\,dy .
\end{equation}
As \(u\to0\), a complete interval between consecutive allowed values
\(s=\sqrt{r/L_k}\), \(s'=\sqrt{(r+1)/L_k}\) contributes
\((s+s')(s'-s)^3/24=(96L_k^2r)^{-1}+O(L_k^{-2}r^{-2})\)
to the inner integral for \(r\ge1\). The variable \(L_ky^2\) ranges from
\(a=L_k\theta_k^2/(u^2(k+1))\) to \(a(1+1/k)\). The partial intervals
at the ends contribute \(O(L_k^{-2}/a)\). Summing the complete
intervals and comparing the harmonic sum with its integral yields
\((96L_k^2)^{-1}(\log(1+1/k)+O(a^{-1}))\).
The relative error is \(O(u^2/L_k)\) since \(a\gg kL_k/u^2\).
Substitution in \eqref{eq:kappa-shell-integral} and summation give
the asserted large-density formula.

As \(\rho\downarrow0\), put \(s=\log(1/\rho)\),
\(K_-=\lfloor s/4\rfloor\) and \(K_+=\lceil4s\rceil\).
Zero is a minimizer if \(4L_k\theta_k^2t/u^2\le1\), which holds
for \(k\le K_-\) since \(\theta_k^2t\ll k\) and \(\log L_k\sim k\).
Choosing zero for \(k\le B\) gives total
\[
 \frac{1}{2u}\sum_{k\le B}\theta_k^2(k^{-2}-(k+1)^{-2})
 =\frac{\log B}{2u}+O(u^{-1}),
\]
by \(\theta_k^2=k^2/2+O(k\log k)\).
This is a lower bound with \(B=K_-\), and an upper bound for
the range \(k\le K_+\) with \(B=K_+\).
By \eqref{eq:grid}, the remaining terms total
\(O(u^3\sum_{k>K_+}k^{-3}L_k^{-2})=O(u^3e^{-K_+})=o(u^{-1})\).
Now \(\log K_\pm=\log s+O(1)\) and \(1/(2u)=\sqrt{\rho/2}\).
\end{proof}

\begin{remark}
The integrals on each \(I_k\) split into elementary integrals at
the midpoint ties. The bound \(G\le u^3/(4\theta_k^2L_k^2t)\) gives the
tails after summation through \(k=B\):
\[
 \frac{u^3}{4}\sum_{k>B}\frac{\log(1+1/k)}{\theta_k^2L_k^2},
 \qquad
 \frac{u^3}{8}\sum_{k>B}
   \frac{\log^2(k+1)-\log^2 k}{\theta_k^2L_k^2}
\]
for \(J,J_1\), respectively. For \(J'\), use \(3/u\) times the first
bound plus \(u^2\sum_{k>B}(\theta_k^2L_k^2)^{-1}\) from
\eqref{eq:j-endpoints}. At \(u=1/\sqrt2\), \(B=10\), the inequalities
\(\theta_k^2\ge\zeta(2)k\) and \(L_k\ge\max(k,27720)\) give respective
bounds \(7.1\cdot10^{-12}\), \(2.4\cdot10^{-11}\) and \(3.4\cdot10^{-9}\).
\end{remark}

\section{Consecutive counts and prime denominators}
\label{sec:exchange}

\subsection{A counting identity and estimates for one prime}

Let \(N\ge4\) and let \(p\in(N/2,N]\) be prime. Here \(B_p\) consists of
the fractions \(a/p\) with \(1\le a\le N\) and \(a\ne p\). Let \(a_p(j)\)
count configurations of total \(j\) using these fractions alone. For every
integer \(j\ge0\), \(a_p(j)\ge1\) by using \(pj\) copies of \(1/p\).
Since the other denominators are prime to \(p\), integrality of \(S\)
forces \(V_p\in\Z\). For large \(N\), these \(V_p\) are independent under
\(\Pp_\beta(\,\cdot\mid\mathcal E_N)\), with masses proportional to
\(q_p^\beta(j)=a_p(j)e^{-\beta j}\), \(j\ge0\).

\begin{lemma}[Counting identity]
\label{lem:block-exchange}
For integers \(m\ge1\) and \(0\le r\le m\),
\begin{equation}
 \frac{R_N(m+1)}{R_N(m)}
 =\frac{a_p(r+1)}{a_p(r)}
   \frac{\mathbb U_{N,m}(V_p=r)}{\mathbb U_{N,m+1}(V_p=r+1)} .
 \label{eq:block-exchange}
\end{equation}
For integers \(0\le j\le m\) and \(r\ge0\),
\begin{equation}
 \mathbb U_{N,m}(V_p=j)
 \le \frac{a_p(j)}{a_p(r)}\frac{R_N(m-j+r)}{R_N(m)} .
 \label{eq:block-exchange-bound}
\end{equation}
\end{lemma}

\begin{proof}
Write \(R_N^{(p)}(s)\) for the count with denominator \(p\) omitted. Then
\(R_N(m)\mathbb U_{N,m}(V_p=j)=a_p(j)R_N^{(p)}(m-j)\). The masses in
\eqref{eq:block-exchange} share the factor \(R_N^{(p)}(m-r)>0\), since
\(1/2\) is available. The inequality follows from
\(a_p(r)R_N^{(p)}(m-j)\le R_N(m-j+r)\).
\end{proof}

Put \(C_0=\pi\sqrt{2/3}\), \(\Delta_j=\sqrt{j+1}-\sqrt j\), and
\(e_p(j)=\log(a_p(j+1)/a_p(j))\).
The ordinary partition function \(\mathfrak p\) satisfies, for \(j\ge1\),
\[
 0\le\mathfrak p(pj)-a_p(j)\le(1+pj)\mathfrak p(p(j-1)).
\]
For each prohibited size \(p\le a\le pj\), removing one such part
injects the partitions containing it into those of \(pj-a\). Since \(\mathfrak p\) is nondecreasing,
this proves the bound. The Hardy--Ramanujan formula \cite{hardy-ramanujan},
\(\mathfrak p(s)=e^{C_0\sqrt s}(4\sqrt3\,s)^{-1}(1+O(s^{-1/2}))\),
shows that the excluded partitions account for an exponentially small
fraction of \(\mathfrak p(pj)\) in \(\sqrt p\) for bounded positive \(j\).
It follows that
\begin{align}
 e_p(j)&=C_0\Delta_j\sqrt p-\log((j+1)/j)+O_J(N^{-1/2})
                 &&(1\le j\le J),\label{eq:block-profile-refined}\\
 e_p(0)&=C_0\sqrt p-\log(4\sqrt3\,p)+O(N^{-1/2}).\nonumber
\end{align}
Thus, for fixed \(J\),
\begin{equation}
 e_p(j)=C_0\Delta_j\sqrt p+O_J(\log N)\quad(0\le j\le J),
 \qquad a_p(j)\le e^{C_0\sqrt{pj}}\quad(j\ge0).
 \label{eq:block-profile}
\end{equation}
For the latter bound, \(\log\prod_{a\ge1}(1-e^{-ba})^{-1}
\le b^{-1}\int_0^\infty-\log(1-e^{-x})\,dx=\pi^2/(6b)\).
Minimizing \(bpj+\pi^2/(6b)\) over \(b>0\) gives the result.

For \(j\ge0\) let \(t_j(u)=(u/(C_0\Delta_j))^2\) and
\(\rho_j=t_j(1/\sqrt2)\), as in Proposition~\ref{prop:kappa-corners}.
Set \(t_{-1}(u)=0\). For \(1/2<t\le1\), let
\(d=d(u,t)=\min_{j\ge0}|t-t_j(u)|\) and let \(\omega=\omega(u,t)\)
satisfy \(t_{\omega-1}(u)<t\le t_\omega(u)\). This integer is the
smallest maximizer of \(C_0\sqrt{tj}-uj\) over \(j\in\Z_{\ge0}\).

\begin{lemma}
\label{lem:block-profile-gap}
For \(u\) in a compact positive interval, \(\beta=u\sqrt N\) and \(t=p/N\),
there are \(c,D>0\) such that \(d\sqrt N\ge D\log N\) implies
\begin{equation}
 \frac{q_p^\beta(j)}{q_p^\beta(\omega)}
 \le \exp(-cd\sqrt N\,|j-\omega|)\qquad(j\ge0).
 \label{eq:block-profile-gap}
\end{equation}
\end{lemma}

\begin{proof}
For bounded \(j\), the log ratio of consecutive weights is
\(e_p(j)-\beta=C_0\Delta_j\sqrt N(\sqrt t-\sqrt{t_j(u)})+O(\log N)\).
For \(D\) large it has magnitude \(\gg d\sqrt N\) and is positive for
\(j<\omega\) and negative for \(j\ge\omega\). Multiplying the bounds gives
\eqref{eq:block-profile-gap} on bounded index ranges and
\(q_p^\beta(\omega)\ge q_p^\beta(0)=1\). For sufficiently large \(j\),
uniformly in \(u,t\), \eqref{eq:block-profile} gives
\(q_p^\beta(j)\le e^{-\beta j/2}\). Since \(d\) and \(\omega\) are bounded,
this also gives \eqref{eq:block-profile-gap}, with a smaller \(c\).
\end{proof}

\subsection{Ratios and simultaneous concentration}

\begin{theorem}
\label{thm:ratio}
Uniformly for \(\rho=m/N\) in a compact subset
\(I\subset(\rho_0,2\rho_0)\cup(\rho_1,\infty)\),
\begin{align}
\label{eq:ratio-at-saddle}
 \log\frac{R_N(m+1)}{R_N(m)}
 &=\widehat\beta_{N,m}+O_I(N^{1/4}\sqrt{\log N}),\\
\label{eq:ratio-two-term}
 \log\frac{R_N(m+1)}{R_N(m)}
 &=\sqrt N\left(u_\rho-\frac{\kappa'(\rho)}{\log N}
                            +O_I((\log N)^{-2})\right).
\end{align}
\end{theorem}

\begin{proof}
Set \(\beta=\widehat\beta_{N,m}\) and \(u=\beta/\sqrt N\).
Since \(t_r(u_\rho)=\rho_r/\rho\) and \(\rho_{r+1}/\rho_r<(r+1)/r\)
for \(r\ge1\), the intervals \((\rho_r,2\rho_r)\) cover the asserted
density range. By compactness and \eqref{eq:saddle-expansion}, there are
\(\eta>0\) and a bounded index \(r\) with \(1/2+\eta\le t_r(u)\le1-\eta\).

Set \(p_*=Nt_r(u)\) and \(\delta=B N^{-1/4}\sqrt{\log N}\) for fixed
\(B>0\). Let \(W_+\), \(W_-\) be the sets of primes respectively in
\((p_*+\delta N,p_*+2\delta N]\) and \((p_*-2\delta N,p_*-\delta N]\).
The short-interval prime number theorem
\(\sum_{x<j\le x+h}\Lambda(j)\sim h\) uniformly for
\(x^{7/12+\varepsilon}\le h\le x\)
\cite[Section~1]{huxley}\cite[Theorem~1]{ingham}\cite[(1.1)]{matomaki-teravainen}
applies with \(0<\varepsilon<1/6\), since \(x\asymp N\) and \(h=\delta N\).
Prime powers with exponent
at least two contribute \(O(\sqrt N\log^2 N)=o(h)\). Thus
\(|W_\pm|\sim\delta N/\log N\).

Set \(\beta_+=\widehat\beta_{N,m+1}\) and \(\beta_-=\beta\).
At these parameters, the prescribed values \(\omega\) in \(W_+\) and
\(W_-\) are \(r+1\) and \(r\), respectively. Indeed, adjacent transition points are
uniformly separated, and \eqref{eq:saddle-difference} moves each
\(t_j(\beta/\sqrt N)\) by \(O(N^{-1})\) as \(m\) increases to \(m+1\).
In each window \(d\gg\delta\), so summing \eqref{eq:block-profile-gap}
gives an error probability at each prime of at most \(C e^{-c\delta\sqrt N}\)
under the indicated Boltzmann law conditional on \(\mathcal E_N\).
For \(A_\pm\) the event that at least a quarter of the primes fail the
prescription, independence gives
\[
 \Pp_{\beta_\pm}(A_\pm\mid\mathcal E_N)
 \le 2^{|W_\pm|}(Ce^{-c\delta\sqrt N})^{\lceil|W_\pm|/4\rceil}
 \le e^{-c'B^2N}.
\]
Here \(c'>0\) is independent of fixed \(B\), since \(\delta\sqrt N\to\infty\)
and \(\delta\sqrt N\,|W_\pm|\sim B^2N\).
Choose \(B\) large so \eqref{eq:transfer} bounds each error event by
\(e^{-N}\) under its corresponding uniform law. The expected fraction
of exceptions is at most \(1/4+e^{-N}\). Hence there is a prime in each
window, denoted \(p_\pm\), at which the prescribed value has uniform
probability at least \(1/2\).

Identity \eqref{eq:block-exchange} bounds the ratio of counts between
\(\tfrac12e^{e_{p_-}(r)}\) and \(2e^{e_{p_+}(r)}\).
By \eqref{eq:block-profile}, both exponents are
\(C_0\Delta_r\sqrt{p_*}+O(\delta\sqrt N)=\beta+O(\delta\sqrt N)\).
This proves \eqref{eq:ratio-at-saddle}, and \eqref{eq:saddle-expansion}
gives \eqref{eq:ratio-two-term}.
\end{proof}

\begin{theorem}
\label{thm:top-concentration}
For \(I\) as in Theorem~\ref{thm:ratio}, there are \(B,c>0\) such that,
uniformly for \(m/N\in I\), \(u=\widehat\beta_{N,m}/\sqrt N\), and primes
\(p\in(N/2,N]\) with \(d(u,p/N)\ge B N^{-1/4}\sqrt{\log N}\),
\[
 \mathbb U_{N,m}\bigl(V_p\ne\omega(u,p/N)\bigr)\le e^{-c d\sqrt N}.
\]
For all large \(N\), \(V_p=\omega(u,p/N)\) holds at all such primes
simultaneously with probability at least \(1-e^{-N^{1/4}}\).
\end{theorem}

\begin{proof}
For small fixed \(\eta>0\), all densities within \(\eta\) of \(I\)
remain in the density range of Theorem~\ref{thm:ratio}.
By monotonicity of \(\beta=\widehat\beta_{N,m}\) in \(m\), telescoping
\eqref{eq:ratio-at-saddle} gives
\(R_N(m+h)/R_N(m)\le e^{h\beta+|h|E}\) for integers \(|h|\le\eta N\),
where \(E=O(N^{1/4}\sqrt{\log N})\) is positive.
For \(0\le j\le m\) and \(0<|j-\omega|\le\eta N\),
\eqref{eq:block-exchange-bound} and \eqref{eq:block-profile-gap} give
\[
 \mathbb U_{N,m}(V_p=j)
 \le e^{E|j-\omega|}\frac{q_p^\beta(j)}{q_p^\beta(\omega)}
 \le e^{-c_1d\sqrt N|j-\omega|}
\]
on taking \(B\) large. Since \(\omega\) is bounded, the other values
have \(j-\omega>\eta N\). Adding two copies of \(1/2\) gives
\(R_N(m-j+\omega)\le R_N(m-\lfloor\eta N\rfloor)\).
The bounds \eqref{eq:block-profile} and \eqref{eq:block-exchange-bound}
then give probability at most
\(\exp(C_0\sqrt{pj}-\lfloor\eta N\rfloor(\beta-E))\le e^{-c_2N^{3/2}}\).
Sum over \(j\) for the single-prime bound and then over at most \(N\)
primes for the simultaneous assertion.
\end{proof}

\begin{corollary}
\label{cor:diagonal-concentration}
Let \(t_*=\rho_2\). For fixed \(\varepsilon>0\) and large \(N\),
under \(\mathbb U_{N,N}\) the following hold simultaneously, with
probability at least \(1-e^{-c_\varepsilon\sqrt N}\):
\[
 V_p=2\quad(1/2<p/N\le t_*-\varepsilon),\qquad
 V_p=3\quad(t_*+\varepsilon\le p/N\le1).
\]
For fixed \(\rho\in(\rho_0,2\rho_0)\) and \(m=\lfloor\rho N\rfloor\),
the same statement under \(\mathbb U_{N,m}\) holds with \(t_*=\rho_0/\rho\),
the two values replaced by \(0,1\), and the constant \(c_{\rho,\varepsilon}>0\).
\end{corollary}

\begin{proof}
At \(u=\widehat\beta_{N,m}/\sqrt N\), \eqref{eq:saddle-expansion} gives
\(t_j(u)\to \rho_j/\rho\). For \(\rho=1\),
\(\rho_1<1/2<\rho_2<1<\rho_3\). In the lower density range only
\(t_0(u)\) has its limit in \((1/2,1)\). In both cases the prescribed
ranges have \(d\gg_{\rho,\varepsilon}1\). Apply
Theorem~\ref{thm:top-concentration} and sum over primes.
\end{proof}

\subsection{The transition and other prime ranges}

\begin{proposition}\label{prop:logistic-block}
Let \(u\) range over a compact positive interval, \(\beta=u\sqrt N\),
and let \(r\) be an integer with \(0\le r\le J\) for fixed \(J\).
For fixed \(\eta>0\), suppose
\(t_{r-1}(u)+\eta\le p/N\le t_{r+1}(u)-\eta\). Then for some \(c>0\),
\begin{align*}
 \Pp_\beta(V_p=r+1\mid\mathcal E_N)
   &=\frac{1+O(e^{-c\sqrt N})}{1+e^{\beta-e_p(r)}},\\
 \Pp_\beta(V_p=r\mid\mathcal E_N)
   &=\frac{1+O(e^{-c\sqrt N})}{1+e^{e_p(r)-\beta}}.
\end{align*}
\end{proposition}

\begin{proof}
By separation from the adjacent transitions and \eqref{eq:block-profile},
there is \(c_1>0\) such that \(e_p(j)-\beta\ge c_1\sqrt N\) for \(j<r\)
and \(e_p(j)-\beta\le-c_1\sqrt N\) for bounded \(j\ge r+1\).
The positive increments give \(q_p^\beta(r)\ge q_p^\beta(0)=1\).
For large \(j\), use \(q_p^\beta(j)\le e^{-\beta j/2}\).
Thus the total weight outside \(\{r,r+1\}\) is \(O(e^{-c\sqrt N})\)
times the sum of the two weights. Their exact ratio is \(e^{e_p(r)-\beta}\).
\end{proof}

Write \(p_0=Nt_r(u)\), with \(p_0/N\) in a compact subset of \((1/2,1)\).
For \(r\ge1\), \eqref{eq:block-profile-refined} gives
\begin{equation}\label{eq:logistic-exponent}
 \beta-e_p(r)=-C_0\Delta_r(\sqrt p-\sqrt{p_0})
                   +\log((r+1)/r)+O(N^{-1/2}).
\end{equation}
The real center of the approximation is the zero after dropping the error, given by
\(\sqrt p=\sqrt{p_0}+\log((r+1)/r)/(C_0\Delta_r)\).
At the diagonal saddle, \eqref{eq:saddle-expansion} gives
\[
 \frac{p_0}{N}=t_*\left(1-\frac{2\sqrt2\,\kappa'(1)}{\log N}
                         +O((\log N)^{-2})\right)\quad(r=2).
\]
Here \(\kappa'(1)\approx-0.0052673\) by Lemma~\ref{lem:prime-expansion}.
The center for \(r=2\) is \(p_0\) times
\(1+(2\sqrt2\log(3/2)+o(1))/\sqrt N\).
For \(r=0\), the center is the real solution on \([N/2,N]\) of
\(C_0(\sqrt p-\sqrt{p_0})=\log(4\sqrt3\,p)\), with
\[
 p=p_0+\frac{2\sqrt{p_0}}{C_0}\log(4\sqrt3\,p_0)+O((\log N)^2).
\]
The difference of the two sides has derivative \(\asymp N^{-1/2}\)
and opposite signs at the endpoints. Substituting the two explicit
terms leaves a residual \(O((\log N)^2/\sqrt N)\). The log odds vary on
the scale \(\sqrt N\) in \(p\) about these real centers.

For \(m\ge1\) and \(0\le r\le m\), \eqref{eq:block-exchange} gives
\begin{equation}\label{eq:adjacent-block-odds}
 \frac{\mathbb U_{N,m+1}(V_p=r+1)}{\mathbb U_{N,m}(V_p=r)}
 =\exp\!\left(e_p(r)-\log\frac{R_N(m+1)}{R_N(m)}\right).
\end{equation}
A logistic limit within one uniform law remains open. For example,
at \(N=100,p=83\), the coefficient
\(a_p(j)=[z^{pj}]\prod_{1\le a\le N,\,a\ne p}(1-z^a)^{-1}\)
gives \(a_{83}(2)=189297163413\) and \(a_{83}(3)=213101829513341\).
Using the procedure in Section~\ref{sec:conjecture}, the log ratio
\(\log(R_{100}(101)/R_{100}(100))\) is about \(7.0372\), compared
with \(7.0825\) from the two leading terms of \eqref{eq:ratio-two-term}.
For \(m=100\) and \(r=2\), the adjacent-total ratio in
\eqref{eq:adjacent-block-odds} is about \(0.9891\).
At total \(100\), the probabilities of \(V_{83}=2\) and \(V_{83}=3\)
are about \(0.4013985\) and \(0.3925051\), respectively, by using
the complementary counts at totals \(98\) and \(97\).

For \(k\ge1\), \(\rho>0\) and \(t\in(1/(k+1),1/k]\), let
\(\omega_{k,\rho}(t)\) be the unique
minimizer in \eqref{eq:G} at \(u=u_\rho\), when it exists. The minimum
is attained twice at the values \(t=t_r(u_\rho)/(\Phi_kL_k)\) in this
interval, corresponding to midpoints between successive \(\sqrt{r/L_k}\).

\begin{theorem}
\label{thm:fixed-shells}
Fix \(k\ge1\) and \(\rho>0\), and let \(T\subset(1/(k+1),1/k]\) be a closed
interval where \(\omega_{k,\rho}\) is defined. Let \(M_N\) count the primes
with \(p/N\in T\), and \(Y_N\) count those with \(V_p\ne\omega_{k,\rho}(p/N)\).
For \(m=\lfloor\rho N\rfloor\), there are \(c,C>0\) such that for all
large \(N\) and \(1\le r\le M_N\),
\begin{equation}
 \mathbb U_{N,m}(Y_N\ge r)\le e^{CN}\binom{M_N}{r}e^{-cr\sqrt N}.
 \label{eq:shell-deviations}
\end{equation}
In particular, \(Y_N\le A\sqrt N\) with probability at least \(1-e^{-c'N}\)
for suitable \(A,c'>0\).
\end{theorem}

\begin{proof}
On \(T\) the minimum is separated uniformly from the other values,
since only finitely many \(w\) can approach the minimum.
The gap and minimizer persist at
\(u=\widehat\beta_{N,m}/\sqrt N\to u_\rho\).
Dividing \eqref{eq:bounded-block-upper} by \eqref{eq:winning-block-lower}
gives a conditional error probability at most \(e^{-c\sqrt N}\) after
summing over \(w\le16w_0\). The geometric tail in
Proposition~\ref{prop:block} gives the same bound for \(w>16w_0\), using
\(M\le\beta w_0\) and \(w_0\asymp1\) for fixed \(k\).
All the primes belong to \(\mathcal S_N\) for large \(N\). Independence,
a union bound, and \eqref{eq:transfer} give \eqref{eq:shell-deviations}.
At \(r=\lfloor A\sqrt N\rfloor+1\) the logarithm of the binomial factor is
\(O_A(\sqrt N\log N)\), so large \(A\) gives the asserted bound.
If \(r>M_N\), the failure event is empty.
\end{proof}

For \(k=1\) and densities covered by Theorem~\ref{thm:top-concentration},
that theorem gives \(Y_N=0\) on such a fixed \(T\) with probability at
least \(1-e^{-c\sqrt N}\).

\section{Denominator and size in a uniform partition}\label{sec:sampling}

Fix \(\rho>0\), \(m=\lfloor\rho N\rfloor\), and \(u=(2\rho)^{-1/2}\).
For \(\alpha=a/b\in\CN\) put \(t_\alpha=b/N\) and \(x_\alpha=\sqrt N\,\alpha\).
For \(N\) sufficiently large that \(m>0\), sample a fraction \(\alpha\)
from a uniform partition of \(m\) with probability \(\alpha X_\alpha/m\)
or uniformly among the \(D_N\) distinct fractions present. These are the size-biased and distinct-part sampling
procedures of \cite{corteel-et-al}. Their limiting size marginals agree
after rescaling with \cite[Theorems~2--3]{mutafchiev}.

\begin{theorem}\label{thm:sampling}
The two conditional laws of \((t_\alpha,x_\alpha)\), given the partition,
converge in probability against bounded continuous tests on
\([0,1]\times[0,\infty)\) to product laws with respective densities
on \((0,1)\times(0,\infty)\)
\begin{equation}\label{eq:sampling-limits}
  (2t)\frac{6u^2}{\pi^2}\frac{x}{e^{ux}-1},
  \qquad (2t)u e^{-ux}.
\end{equation}
These are also the unconditional weak limits. Moreover,
\(D_N/N^{3/2}\to 3/(\pi^2u)\) in probability and in \(L^q\) for every
fixed \(1\le q<\infty\).
\end{theorem}

\begin{lemma}\label{lem:sampling-coprime}
For continuous \(H\) on \([0,1]\times[0,\infty)\) with
\(|H(t,x)|\le C e^{-cx}\) for some \(C,c>0\),
\begin{equation}\label{eq:sampling-coprime}
 N^{-3/2}\sum_{\alpha\in\CN} H(t_\alpha,x_\alpha)
 \longrightarrow \frac{6}{\pi^2}\int_0^1\int_0^\infty tH(t,x)\,dx\,dt .
\end{equation}
\end{lemma}

\begin{proof}
For \(H\) supported on \(x\le A\), write the indicator of coprimality
as \(\sum_{d\mid a,b}\mu(d)\). For fixed \(d\), the mesh
\((a/\sqrt N,b/N)\) with \(d\mid a,b\) gives the normalized Riemann
integral
\[
 d^{-2}\int_0^1\int_0^{At} H(t,s/t)\,ds\,dt .
\]
The bounded integrand is continuous off the boundary of the triangle.
On its support, \(a\le A\sqrt N\) and \(b=1\) would require
\(a\le A/\sqrt N\), so the restrictions \(a\le N\) and \(b\ge2\) cause
no additional boundary terms for large \(N\). The absolute contribution
from \(d>D\) is at most \(A\|H\|_\infty\sum_{d>D}d^{-2}\), by counting
multiples of \(d\) with \(a\le A\sqrt N\) and \(b\le N\). Use \(s=tx\)
and \(\sum_d\mu(d)/d^2=6/\pi^2\). Finally for any \(c'>0\),
\[
 \sum_{b=2}^N\sum_{a\ge1}e^{-c'\sqrt N a/b}
 \le\sum_{b=2}^N\frac{b}{c'\sqrt N}=O_{c'}(N^{3/2}).
\]
With \(c'=c/2\), this bounds the normalized tail \(x_\alpha>A\) by
\(O(e^{-cA/2})\). A continuous cutoff gives the result.
\end{proof}

\begin{proof}[Proof of Theorem~\ref{thm:sampling}]
For a bounded continuous \(f\), write \(f_\alpha=f(t_\alpha,x_\alpha)\)
and set
\[
 M_N(f)=N^{-1}\sum_\alpha\alpha X_\alpha f_\alpha,\qquad
 O_N(f)=N^{-3/2}\sum_\alpha\mathbf 1_{\{X_\alpha>0\}}f_\alpha .
\]
Under \(\mathbb P_{u\sqrt N}\), set \(Z=M_N(f)\) or \(O_N(f)\).
Then \(N^{-3/2}\log\mathbb E e^{\theta N^{3/2}Z}\) has the form on the
left of \eqref{eq:sampling-coprime}, with respective summands
\[
 \log\frac{1-e^{-ux}}{1-e^{-x(u-\theta f(t,x))}},\qquad
 \log\bigl(1+e^{-ux}(e^{\theta f(t,x)}-1)\bigr).
\]
For \(|\theta|\|f\|_\infty\le u/2\) they extend continuously to \(x=0\).
Their \(\theta\)-derivatives are bounded respectively by
\(\|f\|_\infty x/(e^{ux/2}-1)\) and \(O_{f,u}(e^{-ux})\).
Thus Lemma~\ref{lem:sampling-coprime} applies, and the limiting log
moment generating functions are differentiable at zero, with derivatives
\[
 \frac{6}{\pi^2}\int_0^1\int_0^\infty
   \frac{tx f(t,x)}{e^{ux}-1}\,dx\,dt,\qquad
 \frac{6}{\pi^2}\int_0^1\int_0^\infty t e^{-ux}f(t,x)\,dx\,dt .
\]
For either statistic, Chernoff's inequality at sufficiently small positive
and negative \(\theta\) bounds the probability of any fixed positive
deviation from its limit by \(2e^{-c_{f,\rho,\varepsilon}N^{3/2}}\).
By \eqref{eq:boltzmann}, Lemma~\ref{lem:bulk}, and
Theorem~\ref{thm:enumeration},
\(\log\mathbb P_{u\sqrt N}(S=m)=O_\rho(N^{3/2}/\log N)\), because
\(\rho u+1/(2u)=\sqrt{2\rho}\). Dividing by this conditioning probability
still leaves exponentially small deviation bounds on the \(N^{3/2}\) scale.

The two sampling expectations are \(NM_N(f)/m\) and \(O_N(f)/O_N(1)\).
Taking \(f=1\) gives \(O_N(1)=D_N/N^{3/2}\to3/(\pi^2u)\). The
deviation bound and \(D_N\le N^2\) give convergence in each stated \(L^q\):
on deviations of size more than \(\varepsilon\) the \(q\)-th power is
polynomially bounded, and elsewhere it is at most \(\varepsilon^q\).
These limits give \eqref{eq:sampling-limits}; boundedness of the test
function yields unconditional convergence by taking expectations.
\end{proof}

\section{A fixed bound on numerators and denominators}\label{sec:residue}

\subsection{The generating function}

Write each multiplicity as \(b k_{a/b}+r_{a/b}\) with \(0\le r_{a/b}<b\)
and put \(w(r)=\sum_{a/b\in\CN}(a/b)r_{a/b}\). Since the quotients
contribute integers, the generating function is
\begin{equation}\label{eq:residue-numerator}
 \sum_{m\ge0} R_N(m)z^m
  =\frac{A_N(z)}{\prod_{a/b\in\CN}(1-z^a)},\qquad
 A_N(z)=\sum_{r:\,w(r)\in\Z}z^{w(r)},
\end{equation}
where the sum runs over the indicated residue ranges. This is also
the fundamental parallelepiped decomposition for the cone with
primitive rays \((b e_{a/b},a)\), with \(e_{a/b}\) the coordinate vector
\cite[Theorem~3.5]{beck-robins}.

\begin{lemma}\label{lem:residue-map}
The homomorphism \(r\mapsto w(r)\bmod1\) from
\(\prod_{a/b\in\CN}\Z/b\Z\) to \((L_N^{-1}\Z)/\Z\) is onto.
Consequently \(A_N(1)=L_N^{-1}\prod_{a/b\in\CN}b\) and
\(\deg A_N<\sum_{a/b\in\CN}a\).
\end{lemma}

\begin{proof}
For each prime \(p\le N\), the coordinate indexed by \(1/q_p\) generates the subgroup of order \(q_p\).
These subgroups generate the target, so the kernel has the stated size.
The degree bound follows from \(w(r)\le\sum_{a/b}(a/b)(b-1)<\sum_{a/b}a\).
\end{proof}

\begin{theorem}\label{thm:quasipolynomial}
Let \(n=|\CN|\) and let \(e_N\) be the largest number of fractions in
\(\CN\) whose numerators are divisible by a common prime. Then
\[
 R_N(m)=\sum_{j=0}^{n-1}c_j(m)m^j,\qquad
 c_{n-1}=\frac{1}{L_N(n-1)!\prod_{\alpha\in\CN}\alpha}
 \quad(m\ge0),
\]
where each \(c_j\) has period dividing \(L_N\) and \(c_j\) is constant
for \(j\ge e_N\).
\end{theorem}

\begin{proof}
By Lemma~\ref{lem:residue-map}, \eqref{eq:residue-numerator} is proper.
A pole \(\zeta\) of order \(r\) contributes \(\zeta^{-m}\) times a
polynomial of degree at most \(r-1\), and \(\zeta^{L_N}=1\).
At \(\zeta=1\) the leading term is
\(A_N(1)(1-z)^{-n}/\prod_{a/b\in\CN}a\).
For \(\zeta\ne1\), any fixed prime divisor of its multiplicative order divides
every \(a\) with \(\zeta^a=1\), so \(r\le e_N\).
\end{proof}

\begin{remark}\label{ex:fixed-small}
For \(N=3\), \(A_3(z)=(1+z^2)(1+z+z^2)\), giving
\[
 R_3(m)=\frac{m^3}{6}+\frac{3m^2}{4}+\frac{4m}{3}+
        \frac78+\frac{(-1)^m}{8}\quad(m\ge0).
\]
For \(N=4\), \(e_4=2\) and direct enumeration gives
\(A_4(z)=(1+z)(1+z+z^2)(1+z+6z^2+3z^3+9z^4+2z^5+2z^6)\).
The factors reduce the pole orders at \(-1\) and the primitive
cube roots to at most one. Primitive fourth roots already have pole
order at most one,
since \(4/3\) is the only part of numerator \(4\). Thus \(c_1\) is constant.
\end{remark}

\begin{remark}\label{rem:fixed-comparison}
For fixed \(N\), let \(T_N(m)\) count configurations with \(m\le S<m+1\).
Every residue vector \(r\) contributes the coefficient at
\(m-\lfloor w(r)\rfloor\) of \(\prod_{a/b\in\CN}(1-z^a)^{-1}\).
Since only the pole at \(1\) has order \(n\),
\[
 T_N(m)\sim\frac{m^{n-1}}{(n-1)!\prod_{\alpha\in\CN}\alpha},\qquad
 R_N(m)/T_N(m)\longrightarrow L_N^{-1}\quad(m\to\infty).
\]
For \(m=\lfloor\rho N\rfloor\), Corollary~\ref{cor:integer-total}
gives the integrality probability under the Boltzmann law at
\(u_\rho\sqrt N\), whose logarithm is asymptotic to
\(-\kappa(\rho)N^{3/2}/\log N\).

For a prime \(\ell\le N\), write \(a=\ell c\) and apply M\"obius inversion
to \((c,b)=1\) with \(\ell\nmid b\). The count, including \(b=1\), is
\[
 \sum_{\substack{d\le N\\ \ell\nmid d}}\mu(d)
  \left\lfloor\frac{N}{\ell d}\right\rfloor
  \left(\left\lfloor\frac Nd\right\rfloor-\left\lfloor\frac{N}{\ell d}\right\rfloor\right)
 =\frac{6N^2}{\pi^2(\ell+1)}+O(N\log N).
\]
The floor errors sum to \(O(N\log N)\) uniformly in \(\ell\), the
Euler product gives \(\sum_{\ell\nmid d}\mu(d)/d^2=
1/(\zeta(2)(1-\ell^{-2}))\), and deleting \(b=1\) costs \(O(N)\).
Thus \(e_N=2N^2/\pi^2+O(N\log N)\). As \(n=6N^2/\pi^2+O(N\log N)\)
by the same coprime pair count, at least
\(n-e_N=(2/3+o(1))n\) coefficients are constant.

For fixed \(N\ge3\) the counts are eventually strictly log-concave.
Indeed \(e_N\le n-3\), using \(1/2,1/3,1/4\) if \(N\ge4\) and \(e_3=1\) if \(N=3\). Thus \(R_N(m)=P(m)+O_N(m^{d-3})\) with
\(d=n-1\) and \(P\) of degree \(d\) and leading coefficient \(c>0\).
Expansion gives
\(R_N(m)^2-R_N(m-1)R_N(m+1)=dc^2m^{2d-2}+O_N(m^{2d-3})\).
For \(N=2\) every count is one.
\end{remark}

\subsection{Residue sums and a central limit theorem}

Let \(\widetilde r_{a/b}\) be independent and uniform on
\(\{0,\ldots,b-1\}\), and put \(\widetilde W_N=w(\widetilde r)\).
Condition on \(\widetilde W_N\in\Z\) to obtain \(r\) and \(W_N=w(r)\).
Then \(\mathbb P(W_N=j)=[z^j]A_N(z)/A_N(1)\). Write \(h_N=\lceil N/2\rceil\).

\begin{theorem}\label{thm:residue-moments}
For \(N\ge4\), any collection of fewer than \(h_N\) coordinates of \(r\)
is independent and uniform on its residue sets, but the \(h_N\) coordinates
with denominator \(q_2\) are dependent. The moments of \(W_N\) and
\(\widetilde W_N\) agree through order \(h_N-1\), and
\begin{equation}\label{eq:residue-first-moment}
 \mathbb E W_N^{h_N}-\mathbb E\widetilde W_N^{h_N}
 =\frac{(-1)^{h_N}h_N!}{(2q_2)^{h_N}}
   \prod_{\substack{1\le a\le N\\a\ {\rm odd}}}a .
\end{equation}
\end{theorem}

\begin{proof}
For each prime \(p\le N\), there are
\(N-\lfloor N/p\rfloor\ge h_N\) coordinates of denominator \(q_p\);
each generates the subgroup of order \(q_p\) in
Lemma~\ref{lem:residue-map}. After deleting fewer than \(h_N\)
coordinates, the map on the remaining coordinates is still onto.
Each assignment of the deleted coordinates therefore has the same
number of completions to the kernel, which proves independence and
uniformity. If the \(q_2\)-coordinates are deleted, every remaining
denominator divides \(L_N/2\). The deleted choice \(r_{1/q_2}=1\)
and all other deleted residues zero cannot be completed. Since the
individual marginals are uniform, these coordinates are dependent.
Expanding a moment of order below \(h_N\) uses only the joint laws
of fewer than \(h_N\) coordinates.

By Lemma~\ref{lem:residue-map}, the conditioning event has probability
\(1/L_N\). Character orthogonality and the geometric sum give
\begin{equation}\label{eq:residue-fourier}
 \mathbb E e^{tW_N}
 =\sum_{j=0}^{L_N-1}\prod_{a/b\in\CN}
 \frac{1-e^{at}}{b(1-e^{2\pi i ja/b}e^{at/b})}.
\end{equation}
The \(j=0\) term gives the independent moments. Each factor has
value \(1\) at zero if \(b\mid j\) and a simple zero otherwise.
If \(q_p\nmid j\) for any odd \(p\le N\), this gives at least
\(N-\lfloor N/p\rfloor>h_N\) zeros. The remaining nonzero indices are
\(j=kL_N/q_2\), \(1\le k<q_2\). All \(h_N\) factors of denominator \(q_2\)
vanish, and at least one more, of denominator \(q_2/2\ge2\), vanishes
unless \(k=q_2/2\). At \(j=L_N/2\) only the \(q_2\)-factors vanish,
each with linear term \(-at/(2q_2)\). Differentiation at order \(h_N\)
gives \eqref{eq:residue-first-moment}.
\end{proof}

For \(N=3\) both means are \(3/2\). The variance of \(W_3\), from \(A_3\),
is \(5/3\), whereas that of \(\widetilde W_3\) is \(215/216\).
The second-moment difference is \(145/216\), while the formula would
give \(3/8\).

\begin{theorem}\label{thm:residue-gaussian}
For \(N\ge5\) the mean and variance of \(W_N\) satisfy
\begin{align*}
 \mu_N&=\tfrac12\sum_{a/b\in\CN}a(1-b^{-1})
       =\frac{3N^3}{2\pi^2}+O(N^2\log N),\\
 \sigma_N^2&=\tfrac1{12}\sum_{a/b\in\CN}a^2(1-b^{-2})
       =\frac{N^4}{6\pi^2}+O(N^3\log N).
\end{align*}
The law of \((W_N-\mu_N)/\sigma_N\) converges to the standard normal
law, with convergence of every fixed moment.
\end{theorem}

\begin{proof}
The first two moments coincide with those of \(\widetilde W_N\) by
Theorem~\ref{thm:residue-moments}, so the exact formulas follow from
the uniform coordinate laws. For \(s=1,2\), M\"obius inversion gives
\[
 \sum_{a/b\in\CN}a^s
 =\sum_{d\le N}\mu(d)d^s\left\lfloor N/d\right\rfloor
    \sum_{v\le N/d}v^s+O(N^{s+1})
 =\frac{N^{s+2}}{(s+1)\zeta(2)}+O(N^{s+1}\log N).
\]
The first error is for \(b=1\), and in the second the power sum errors
are \(O(N^{s+1}\sum_{d\le N}d^{-1})\). Summation without the
coprimality restriction bounds the omitted \(b^{-1}\) and \(b^{-2}\)
terms by \(O(N^2\log N)\) and \(O(N^3)\), respectively.

Each centered independent summand is bounded by \(N/2\), so its
cumulant of a fixed order \(s\) is \(O_s(N^s)\). With \(O(N^2)\)
summands and \(\sigma_N\asymp N^2\), the standardized independent
sum has cumulants of order \(s\ge3\) of size \(O_s(N^{2-s})\).
Thus all its moments converge to the normal moments.
The moments of \((W_N-\mu_N)/\sigma_N\) eventually coincide at each
fixed order since \(h_N\to\infty\). Apply the method of moments
\cite{frechet-shohat}.
\end{proof}

\begin{proposition}\label{prop:residue-palindromic}
For \(N\ge2\), \(A_N(z)\) is palindromic exactly for \(N=2,3\).
\end{proposition}

\begin{proof}
Put \(U=\sum_{a/b\in\CN} a\) and \(\tau=\sum_{a/b\in\CN}a/b\).
The reflection \(r_{a/b}\mapsto b-1-r_{a/b}\) sends \(w\) to
\(U-\tau-w\). If \(\tau\in\Z\) this preserves integrality and maps
the sum \(0\) to the largest possible sum \(U-\tau\), proving
palindromicity. In particular \(\tau=3\) for \(N=3\), and \(A_2=1\).
Conversely for \(N\ge4\), the mean is \((U-\tau)/2\) by
Theorem~\ref{thm:residue-moments}; palindromicity would force
this to be \(\deg(A_N)/2\), hence \(\tau\in\Z\).
For any odd prime \(N/2<p\le N\), the contribution of denominator
\(p\) to \(\tau\) is \(N(N+1)/(2p)-1\), and the other denominators
are prime to \(p\). Integrality would thus require \(p=N\) or
\(p=(N+1)/2\). For \(N\ge17\) the interval contains at least three
primes, since the third Ramanujan prime is \(17\) \cite{sondow}.
For \(N=4\), \(6\le N\le8\), \(9\le N\le12\), and \(13\le N\le16\),
use respectively \(p=3,5,7,11\) to violate the condition.
For \(N=5\) the sum is \(51/4\).
\end{proof}

\section{A conjecture and open problems}\label{sec:conjecture}

Let \(\mathcal S_N^\ast=\{p:\sqrt N<p\le N,\ p\ {\rm prime}\}\) and
\(\mathcal E_N^\ast=\bigcap_{p\in\mathcal S_N^\ast}E_p\). Put
\[
 \Psi_{N,m}^\ast(\beta)
 =\beta m+F_N(\beta)+\sum_{p\in\mathcal S_N^\ast}\log\Pp_\beta(E_p)
 =\log\sum_{x\in\mathcal E_N^\ast}e^{\beta(m-S(x))}.
\]
For \(m\ge1\) this is strictly convex, since its second derivative
is the strictly positive conditional variance. The empty configuration
gives divergence as \(\beta\to\infty\), and all integer totals from
copies of \(1/2\) alone give divergence as \(\beta\downarrow0\).
Thus the unique minimizer \(\beta^\ast\) has conditional mean \(m\).
Write \(\sigma_\ast^2=(\Psi_{N,m}^\ast)''(\beta^\ast)\).
On \(\mathcal E_N^\ast\), \(S\) lies in \(Q_N^{-1}\Z\) and
\(L_N=Q_N\prod_{p\in\mathcal S_N^\ast}p\).
This lattice has maximal span for \(N\ge4\), since the support
contains \(0\) and each \(1/q_\ell\) with \(\ell\le\sqrt N\) prime,
and the corresponding integers \(Q_N/q_\ell\) have greatest common
divisor \(1\). For every \(\beta>0\),
\begin{equation}\label{eq:local-limit-identity}
 R_N(m)=e^{\Psi_{N,m}^\ast(\beta)}
             \Pp_\beta(S=m\mid\mathcal E_N^\ast).
\end{equation}

\begin{conjecture}\label{conj:local-limit}
Uniformly for \(m/N\) in compact positive intervals,
\begin{equation}\label{eq:local-limit-conjecture}
 R_N(m)\sim
 \frac{e^{\Psi_{N,m}^\ast(\beta^\ast)}}{Q_N\sqrt{2\pi\sigma_\ast^2}}
 =\frac{e^{\beta^\ast m+F_N(\beta^\ast)}}{L_N\sqrt{2\pi\sigma_\ast^2}}
   \prod_{p\in\mathcal S_N^\ast}p\,\Pp_{\beta^\ast}(E_p).
\end{equation}
\end{conjecture}

The conjecture is a local limit statement at \(m\) for the law
in \eqref{eq:local-limit-identity}. For every \(\beta>0\) and large \(N\),
the inclusions \(\{S=m\}\subset\mathcal E_N^\ast\subset\mathcal E_N\)
give \(\log R_N(m)\le\Psi_{N,m}^\ast(\beta)\le\Psi_{N,m}(\beta)\).
Taking minima and applying Proposition~\ref{prop:saddle} shows that
\[
 0\le\Psi_{N,m}^\ast(\beta^\ast)-\log R_N(m)\le CN
\]
on compact positive intervals of \(m/N\). The conjecture predicts
\(\log Q_N+\tfrac12\log(2\pi\sigma_\ast^2)+o(1)\) for this difference.
Here \(\log Q_N=2\vartheta(\sqrt N)+O(N^{1/3}\log N)\sim2\sqrt N\),
where \(\vartheta(x)=\sum_{p\le x}\log p\), by counting prime powers
and using the prime number theorem.

Fix a compact positive interval \(I\) for \(m/N\). To bound the variance
at the minimizer, first note \(\beta^\ast\asymp\sqrt N\). Indeed
\(\Psi_{N,m}^\ast(\beta^\ast)\le\Psi_{N,m}^\ast(\sqrt N)\le C_I N^{3/2}\)
by Lemma~\ref{lem:bulk}, and the empty configuration gives
\(\beta^\ast m\le C_I N^{3/2}\).
Choose a fixed \(a>\max I\) with \(\sqrt{2a}>C_I+1\) and let
\(m'=\lfloor aN\rfloor\). The configurations of total \(m'\) give
\(\Psi_{N,m}^\ast(\beta^\ast)\ge\log R_N(m')-\beta^\ast(m'-m)\).
By Theorem~\ref{thm:enumeration} this yields \(\beta^\ast\gg\sqrt N\).
The independent coordinate at \(1/q_2\) contributes variance
\(\asymp(\beta^\ast)^{-2}\asymp N^{-1}\), since \(\beta^\ast/q_2\to0\). In the other direction, Markov's
inequality and the empty configuration give
\[
 \Pp_{\beta^\ast}(S>x\mid\mathcal E_N^\ast)
 \le \exp\{F_N(\beta^\ast/2)-\beta^\ast x/2\}.
\]
Integrating the tail beyond a sufficiently large multiple of \(N\)
gives \(\sigma_\ast^2=O(N^2)\). Hence \(\log\sigma_\ast^2=O(\log N)\)
and the conjectural difference is \(2\sqrt N+o(\sqrt N)\).

For exact computation of \(R_N(m)\), process denominators grouped by
greatest prime factor \(p\), in decreasing order of \(p\).
Starting from \(1\), multiply the formal series in \(z^{1/L_N}\)
by \((1-z^\alpha)^{-1}\) for each fraction
in a group, dropping powers above \(m\). After the group indexed by
\(p\), retain only powers whose reduced denominators are prime to \(p\),
since all later denominators are prime to \(p\). The final coefficient
at \(m\) is \(R_N(m)\). For \(k=\lfloor N/p\rfloor\), the probability
in \eqref{eq:local-limit-conjecture} is
\[
 \Pp_\beta(E_p)=\frac1p\sum_{h=0}^{p-1}
  \prod_{\alpha\in B_p}
  \frac{1-e^{-\beta\alpha}}{1-e^{-\beta\alpha+2\pi i hL_k\alpha}},
\]
by filtering \(pL_kV_p\) modulo \(p\).
On the diagonal the largest absolute relative errors of the
conjectural approximation are \(8.3735\cdot10^{-3}\) at \(N=31\)
for \(20\le N\le53\), and \(6.6600\cdot10^{-4}\) at \(N=63\) for
\(54\le N\le100\). At \(N=100\), the logarithms of the approximation
and exact count are \(1272.9403\) and \(1272.9401\), respectively.
The exact factors \(p\,\Pp_{\beta^\ast}(E_p)\) are close to \(1\) at
these sizes, so the experiments primarily test the local approximation
on \(Q_N^{-1}\Z\).

The following questions remain open.
\begin{enumerate}
\item Prove Conjecture~\ref{conj:local-limit}. One approach is to bound
the conditional characteristic function at frequencies associated
with the prime-power residues of \(Q_NS\) modulo \(Q_N\).
\item Find an explicit approximation to \(\log R_N(m)\) with additive
error \(o(N)\), uniformly on compact positive ranges of \(m/N\).
Under Conjecture~\ref{conj:local-limit}, Appendix~\ref{sec:free-energy} and \(\log L_N\sim N\) give
\[
 \log R_N(m)=\beta^\ast m+\frac{N^2}{2\beta^\ast}-\frac32N+
 \sum_{p\in\mathcal S_N^\ast}\log\!\bigl(p\,\Pp_{\beta^\ast}(E_p)\bigr)+o(N).
\]
This would require estimating the prime sum and conditional saddle.
\item Does \(R_N(m-1)R_N(m+1)\le R_N(m)^2\) hold for every \(N\ge2\)
and \(m\ge1\)? Remark~\ref{rem:fixed-comparison} settles large \(m\) for
fixed \(N\), and exact computations found no violation among the \(4899\)
pairs with \(2\le N\le70\) and \(1\le m<2N\). This inequality and the
counting identity would help control \(V_p\) inside transition intervals.
\item Extend the ratio and simultaneous concentration results to
\(0<\rho\le\rho_0\) and \(2\rho_0\le\rho\le\rho_1\), where no limiting
transition point \(t_j(u_\rho)\) lies in \((1/2,1)\).
\item Can the exceptional primes in Theorem~\ref{thm:fixed-shells}
be eliminated with high probability for \(k\ge2\)?
\end{enumerate}

\appendix
\section{The Boltzmann normalizing constant}\label{sec:free-energy}

\begin{proposition}\label{prop:free-energy-refined}
Uniformly for \(\beta/\sqrt N\) in a compact positive interval,
there is \(c>0\) such that
\begin{align*}
 F_N(\beta)&=\frac{\pi^2(\Phi_N-1)}{6\beta}-\tfrac12\log L_N
              +O(Ne^{-c\sqrt{\log N}})\\
           &=\frac{N^2}{2\beta}-\frac N2+O(Ne^{-c\sqrt{\log N}}).
\end{align*}
\end{proposition}
\begin{proof}
Put \(M(x)=\sum_{d\le x}\mu(d)\) and
\(\nu_N(j)=M(N/j)-\mathbf1_{\{j=1\}}\).
Extend the numerator sums to \(a\ge1\) as in Lemma~\ref{lem:bulk}.
M\"obius inversion and \(b=jd\) give
\(F_N(\beta)=\sum_{j\le N}\nu_N(j)g(\beta/j)+O(1)\).
The Dirichlet convolution identities \(\mu*\operatorname{id}=\phi\),
\(\mu*1=\mathbf1_{\{1\}}\) and \(\mu*\log=\Lambda\) yield
\(\sum_j\nu_N(j)j=\Phi_N-1\), \(\sum_j\nu_N(j)=0\) and
\(\sum_j\nu_N(j)\log j=\log L_N\), with sums over \(j\le N\).
By \eqref{eq:eta} the remainder after the main terms is
\(\sum_{j\le N}\nu_N(j)\xi(j/\beta)+O(1)\), where
\(\xi(x)=g(4\pi^2x)-1/(24x)\).
For \(x\ge1\), \(\xi(x)=O(1/x)\) by exponential decay.
For \(x\le1\), \eqref{eq:eta} gives
\(\xi(x)=\tfrac12\log(2\pi x)-\pi^2x/6+g(1/x)=O(1+|\log x|)\).

For \(j\le N^{3/4}\), use
\(|M(N/j)|\ll(N/j)e^{-c_1\sqrt{\log N}}\)
from the Mertens bound \cite[(6.17)]{montgomery-vaughan}.
The sum of \(j^{-1}|\xi(j/\beta)|\) is \(O(\log^2 N)\)
by splitting at \(j=\beta\), and the subtracted term at \(j=1\)
is \(O(\log N)\). For \(j>N^{3/4}\), the trivial bound
\(|M(N/j)|\le N/j\) gives a tail \(O(N\beta\sum_{j>N^{3/4}}j^{-2})
=O(N^{3/4})\).
Absorbing logarithmic factors proves the first formula. Substitute the
estimate for \(\Phi_N\) from Lemma~\ref{lem:bulk} and
\(\log L_N=N+O(Ne^{-c_2\sqrt{\log N}})\)
\cite[(6.12)]{montgomery-vaughan} for the second.
\end{proof}

\end{document}